\documentclass[12pt]{amsart}

\usepackage[matrix,arrow,curve,frame]{xy}
\usepackage{amsmath,amsthm,amssymb,enumerate}
\usepackage{latexsym}
\usepackage{amscd}
\usepackage[colorlinks=false]{hyperref}
\usepackage{euscript}
\usepackage{appendix}

\newtheorem{thm}{Theorem}[section]

\newtheorem{cor}{Corollary}[section]

\theoremstyle{definition}

\theoremstyle{remark}

\numberwithin{equation}{section}

\def\cE{{\mathcal E}}

\def\cH{{\mathcal H}}

\def\cN{{\mathcal N}}

\def\cS{{\mathcal S}}

\def\cV{{\mathcal V}}
\def\cW{{\mathcal W}}

\def\gg{{\mathfrak g}}

\def\gl{{\mathfrak l}}

\def\go{{\mathfrak o}}
\def\gp{{\mathfrak p}}

\def\gs{{\mathfrak s}}

\newfont{\german}{eufm10}

\begin{document}
\pagestyle{plain}

\title
{Classical freeness of orthosymplectic affine vertex superalgebras}

\author{Thomas Creutzig}
\address{Department Mathematik, FAU Erlangen, Cauerstrasse 11, 91058 Erlangen, Germany}
\email{creutzigt@math.fau.de}
\thanks{T. Creutzig is supported by NSERC Discovery Grant \#RES0048511.}

\author{Andrew R. Linshaw}
\address{Department of Mathematics, University of Denver, C. M. Knudson Hall, 2390 S. York St. Denver, CO 80210}
\email{andrew.linshaw@du.edu}
\thanks{A. Linshaw is supported by NSF Grant DMS-2001484.}

\author{Bailin Song}
\address{School of Mathematical Sciences, University of Science and Technology of China, Jinzhai Road 96, Hefei 230026, Anhui, P. R. China}
\email{bailinso@ustc.edu.cn}
\thanks{B. Song is supported by  NSFC No. 12171447}

\thanks{We are grateful to Shashank Kanade for raising the question of whether the classical freeness of $L_n(\go\gs\gp_{1|2r})$ could be proven using existing technology.}


{\abstract \noindent The question of when a vertex algebra is a quantization of the arc space of its associated scheme has recently received a lot of attention in both the mathematics and physics literature. This property was first studied by Tomoyuki Arakawa and Anne Moreau [Lectures on $\cW$-algebras, Australian Representation Theory Workshop 2016, University of Melbourne], and was given the name \lq\lq classical freeness" by Jethro van Ekeren and Reimundo Heluani in their work on chiral homology [Comm. Math. Phys. 386 (2021), no. 1, 495-550]. Later, it was extended to vertex superalgebras by Hao Li [Eur. J. Math. 7 (2021), 1689-1728]. In this note, we prove the classical freeness of the simple affine vertex superalgebra $L_n(\go\gs\gp_{m|2r})$ for all positive integers $m,n,r$ satisfying $-\frac{m}{2} + r +n+1 > 0$. In particular, it holds for the rational vertex superalgebras $L_n(\go\gs\gp_{1|2r})$ for all positive integers $r,n$.}

\maketitle
\section{Introduction}
We briefly recall the notion of classical freeness of a vertex algebra. First, for any vertex algebra $\cV$, we have {\it Zhu's commutative algebra} $R_{\cV}$ \cite{Z}. It is defined as the vector space quotient of $\cV$ by the span of all elements of the form $a_{(-2)} b$ for all $a,b \in \cV$. Next, there is a canonical decreasing filtration $F^{\bullet} \cV$, such that the associated graded algebra $\text{gr}^F(\cV)$ is a differential graded commutative ring \cite{Li}. In fact, $R_{\cV}$ can be identified with the zeroth graded component of $\text{gr}^F(\cV)$, and $R_{\cV}$ generates $\text{gr}^F(\cV)$ as a differential algebra. If $\cV$ is freely generated by a set of fields $\{\alpha_i\}$, then $\text{gr}^F(\cV)$ is just the differential polynomial algebra generated by $\{\alpha_i\}$. However, if $\cV$ is not freely generated, it is an interesting problem to find all differential algebraic relations in $\text{gr}^F(\cV)$.

Following Arakawa \cite{Ar2,Ar3}, the {\it associated scheme} of $\cV$ is defined to be $X_{\cV} = \text{Spec} \ R_{\cV}$, which is an affine Poisson scheme. The {\it singular support} of $\cV$ is defined to be
$$\text{SS}(\cV) = \text{Spec} \ \text{gr}^F(\cV),$$ 
which is a vertex Poisson scheme. Let $(R_{\cV})_{\infty}$ denote the affine coordinate ring of the arc space $J_{\infty}(X_{\cV})$. By its universal property, there is a surjective homomorphism of differential rings
\begin{equation} \label{eq:map} \Phi: (R_{\cV})_{\infty} \rightarrow \text{gr}^F(\cV).\end{equation} Equivalently, there is a closed embedding $\text{SS}(\cV)\hookrightarrow J_{\infty}(X_{\cV})$.
The question of when \eqref{eq:map} is an isomorphism (i.e. $\cV$ is a quantization of $J_{\infty}(X_{\cV})$), was first considered by Arakawa and Moreau \cite{AM1,AM2}. They also considered the weaker question of when $\cV$ is a quantization of the arc space of the associated {\it variety} of $\cV$ (i.e., \eqref{eq:map} induces an isomorphism of reduced rings), and proved this when $\cV$ is simple and quasi-lisse. In \cite{EH1}, van Ekeren and Heluani call $\cV$ called {\it classically free} if \eqref{eq:map} is an isomorphism, and this property plays an important role in their computations of chiral homology in \cite{EH1,EH2}. It is easy to see that any freely generated vertex algebra is classically free; this includes all free field algebras, universal affine vertex algebras, and universal $\cW$-algebras. If $\cV$ is not freely generated, the phenomenon is much more subtle. Heluani and van Ekeren showed in \cite{EH1} that the Virasoro minimal models $\text{Vir}_{p,q}$ for coprime positive integers $p<q$ are classically free if and only if $p = 2$. They also showed that $L_k(\gs\gl_2)$ is classically free for $k \in \mathbb{N}$. The proofs of these results are quite intricate, but recently the second and third authors have used the invariant theory of arc spaces to prove that $L_k(\gs\gp_{2n})$ is classically free for all positive integers $k,n$ \cite{LS1}. This generalizes and provides a conceptual proof of van Ekeren and Heluani's result in the case $n=1$.

Recently in \cite{L}, Hao Li has extended the notion of classical freeness in a natural way to vertex superalgebras $\cV$. First, for any supercommutative algebra $A$, $A_{\infty}$ is defined by freely adjoining a differential, i.e. it is the initial object in the category of differential superalgebas generated by $A$. Classical freeness of $\cV$ then means that the map $(R_{\cV})_{\infty} \rightarrow \text{gr}^F(\cV)$ is an isomorphism of differential superalgebras. Li proved that the $\cN=1$ superconformal minimal models $L^{\cN=1}_{p,q}$ are classically free if and only if $(p,q) = (2,4k)$ for $k \in \mathbb{N}$, which is analogous to the above result for $\text{Vir}_{p,q}$. In \cite{L} and in joint work with Milas \cite{LM}, Li gave many additional examples of classically free vertex (super)algebras, but most of these involve principal subspaces and are not simple. In this paper, we use the invariant theory of arc spaces to give the first examples of simple affine vertex superalgebras at nongeneric levels which are classically free.

At present, classical freeness is still quite mysterious. There is no clear intuition for why a given vertex algebra should be classically free, and there are also no good techniques for answering the question in a general setting. As is often the case, a hint might be provided by physics. The Virasoro minimal models that are classical free are precisely those that appear in the physics of Argyres-Douglas theories \cite{Ra}, and physicists expect that vertex algebras that come from four-dimensional physics are classical free; for example, Leonardo Rastelli has expressed this expectation at  the workshop on Representation theory, Vertex and Chiral Algebras in Rio de Janeiro in March 2022.  Other examples of chiral algebras of Argyres-Douglas theories are the $\mathcal B_p$-algebras  and the $\mathcal R^{(p)}$-algebras  \cite{C, CRW, ACGY} and there are many more examples beyond these Argyres-Douglas theories, see e.g. \cite{BM, BMPR}. The $\mathcal B_2$-algebra is nothing but a single $\beta\gamma$-system and it is classical free as a freely generated vertex algebra. 
All these algebras have in common that they allow for nice free field realizations, and it might be possible to exploit this in order to prove classical freeness. We consider it a key problem to prove (or disprove) the classical freeness of these vertex algebras coming from higher dimensional physics.

\section{Background}
We will use the same notation and conventions as the paper \cite{LS1} of the second and third authors.

\subsection{Affine vertex algebras} Let $\gg$ be a simple, finite-dimensional Lie (super)algebra. The {\it universal affine vertex algebra} $V^k(\gg)$ is freely generated by fields $X^{\xi}$ which are linear in $\xi \in \gg$ and satisfy 
\begin{equation}
X^{\xi}(z) X^{\eta}(w) \sim k ( \xi, \eta) (z-w)^{-2} + X^{[\xi,\eta]}(w)(z-w)^{-1}.
\end{equation}
Here $(\cdot ,\cdot )$ denotes the normalized Killing form $\frac{1}{2h^{\vee}} \langle \cdot,\cdot \rangle$. For all $k\neq -h^{\vee}$, $V^k(\gg)$ has the Sugawara Virasoro vector
\begin{equation} \label{sugawara} L^{\gg}  = \frac{1}{2(k+h^{\vee})} \sum_{i=1}^n :X^{\xi_i} X^{\xi'_i}: \end{equation} of central charge $c = \frac{k\ \text{sdim}(\gg)}{k+h^{\vee}}$. Here $\xi_i$ runs over a basis of $\gg$, and $\xi'_i$ is the dual basis with respect to $(\cdot,\cdot)$.

As a module over $\widehat{\gg} = \gg[t,t^{-1}] \oplus \mathbb{C}$, $V^k(\gg)$ is isomorphic to the vacuum $\widehat{\gg}$-module. For generic $k$, $V^k(\gg)$ is a simple vertex algebra, but for certain rational values of $k \geq -h^{\vee}$, $V^k(\gg)$ is not simple, and we denote by $L_k(\gg)$ its simple graded quotient.

We shall adopt the following convention in the case $\gg = \go\gs\gp_{m|2n}$. We take the dual Coxeter number to be $$h^{\vee} = \frac{2n+2-m}{2}.$$ The bilinear form on $\go\gs\gp_{m|2n}$ is then normalized so that it restricts to the usual bilinear form on $\gs\gp_{2n}$, and we have the embedding \begin{equation} \label{osp:convention} V^{k}(\gs\gp_{2n}) \otimes V^{-2k}(\gs\go_m) \rightarrow V^k(\go\gs\gp_{m|2n}).\end{equation}

\subsection{$\beta\gamma$ and $bc$-systems} 
The rank $n$ $\beta\gamma$-system $\cS(n)$ is freely generated by even fields $\beta^1,\dots, \beta^n$ and $\gamma^1,\dots, \gamma^n$, satisfying
 \begin{equation} \label{eq:betagammaope} \begin{split} \beta^i(z)\gamma^{j}(w) &\sim \delta_{i,j} (z-w)^{-1},\quad \gamma^{i}(z)\beta^j(w)\sim -\delta_{i,j} (z-w)^{-1},\\  \beta^i(z)\beta^j(w) &\sim 0,\qquad\qquad\qquad \gamma^i(z)\gamma^j (w)\sim 0.\end{split} \end{equation}  It has Virasoro element $L^{\cS} = \frac{1}{2} \sum_{i=1}^n \big(:\beta^{i}\partial\gamma^{i}: - :\partial\beta^{i}\gamma^{i}:\big)$ of central charge $-n$, under which $\beta^{i}$, $\gamma^{i}$ are primary of weight $\frac{1}{2}$. The symplectic group $Sp_{2n}$ is the full automorphism group of $\cS(n)$ preserving $L^{\cS}$. In fact, there is a homomorphism $L_{-1/2}(\gs\gp_{2n}) \rightarrow \cS(n)$ whose zero modes infinitesimally generate the action of $Sp_{2n}$. 

Similarly, the rank $n$ $bc$-system $\cE(n)$ is freely generated by odd fields $b^1,\dots, b^n$ and $c^1,\dots, c^n$, satisfying
\begin{equation} \label{eq:betagammaope} \begin{split} b^i(z) c^{j}(w) &\sim \delta_{i,j} (z-w)^{-1},\quad c^{i}(z) b^j(w)\sim \delta_{i,j} (z-w)^{-1},\\  b^i(z) b^j(w) &\sim 0,\qquad\qquad\qquad c^i(z)c^j (w)\sim 0.\end{split} \end{equation} It has Virasoro element 
$L^{\cE}= \frac{1}{2} \sum_{i=1}^n \big(-:b^{i}\partial c^{i}: + :\partial b^{i} c^{i}:\big)$
 of central charge $n$, under which $b^{i}$, $c^{i}$ are primary of weight $\frac{1}{2}$. The orthogonal group $O_{2n}$ is the full automorphism group of $\cE(n)$ preserving $L^{\cE}$, and there is a homomorphism $L_{1}(\gs\go_{2n}) \rightarrow \cE(n)$ which infinitesimally generates the action of $O_{2n}$.

Next, we recall certain well-known affine algebra actions on these free field algebras which appear as \cite[Eq. 2.28-2.29]{LS1}. 
\begin{equation} \label{s1} V^{-m+r}(\gg\gl_n) \otimes V^n (\gs\gl_{r|m}) \rightarrow \cS(nm) \otimes \cE(nr),\end{equation}
\begin{equation} \label{s2} V^{-\frac{m}{2}+r}(\gs\gp_{2n}) \otimes V^{n}(\go\gs\gp_{m|2r}) \rightarrow \cS(nm) \otimes \cE(2nr),\end{equation}
The image of $V^{-m+r}(\gg\gl_n) \otimes V^n (\gs\gl_{r|m})$ in $\cS(nm) \otimes \cE(nr)$ under \eqref{s1}, which we denote by $\tilde{V}^{-m+r}(\gg\gl_n) \otimes \tilde{V}^n(\gs\gl_{r|m})$, is conformally embedded. Similarly, the image $\tilde{V}^{-\frac{m}{2}+r}(\gs\gp_{2n}) \otimes \tilde{V}^{n}(\go\gs\gp_{m|2r})$ of $V^{-\frac{m}{2}+r}(\gs\gp_{2n}) \otimes V^{n}(\go\gs\gp_{m|2r})$ under \eqref{s2} is conformally embedded.

The following coset vertex algebras were studied in \cite{LS1}:
\begin{equation} \begin{split} & \text{Com}(\tilde{V}^{-m+r}(\gg\gl_n), \cS(nm) \otimes \cE(nr)),
\\ & \text{Com}(\tilde{V}^{-\frac{m}{2}+r}(\gs\gp_{2n}), \cS(nm) \otimes \cE(2nr)).\end{split} \end{equation}
Using the arc space analogues of Weyl's first and second fundamental theorems of invariant theory for $\gs\gl_n$ and $\gs\gp_{2n}$, the following result was proved for all positive integers $n,m,r$:
\begin{enumerate}
\item $\text{Com}(\tilde{V}^{-m+r}(\gg\gl_n), \cS(nm) \otimes \cE(nr)) \cong \tilde{V}^n(\gs\gl_{r|m})$ (\cite[Thm. 5.6]{LS1}),
\item $\text{Com}(\tilde{V}^{-\frac{m}{2}+r}(\gs\gp_{2n}), \cS(nm) \otimes \cE(2nr)) \cong \tilde{V}^{n}(\go\gs\gp_{m|2r})$ (\cite[Thm. 6.3]{LS1}).
\end{enumerate}

In \cite{LS1}, the question of when these cosets were simple was not addressed, but in some cases this can be answered using the following recent result of Arakawa, Kawasetsu, and the first author \cite{ACK}:

\begin{thm}\label{cosetsimplicity} 
Let $\cV$ be a simple vertex (super)algebra which is $\mathbb{N}$-graded by conformal weight, with one-dimensional weight zero space. Suppose that $\cV$ admits an embedding $\tilde{V}^k(\gg) \rightarrow \cV$, where $\tilde{V}^k(\gg)$ is a homomorphic image of $V^k(\gg)$, and $\gg$ is either a simple Lie algebra or $\gg=\mathfrak{osp}_{1|2n}$. If $k+h^\vee \not\in \mathbb{Q}_{\leq 0}$, then 
$\text{Com}(\tilde{V}^k(\gg), \cV)$ is simple.
\end{thm}

\begin{cor} \label{simplicityembeddings} Let $m,n,r$ be positive integers.

\begin{enumerate}
\item For $n\geq 2$, $m\neq r$, and $-m + r +n  > 0$, the image of the map $V^n (\gs\gl_{r|m})  \rightarrow \cS(nm) \otimes \cE(nr)$ is the simple affine vertex superalgebra $L_n (\gs\gl_{r|m})$.
\item For $-\frac{m}{2} + r +n+1 > 0$, the image of the map $V^{n}(\go\gs\gp_{m|2r}) \rightarrow \cS(nm) \otimes \cE(2nr)$ is the simple affine vertex superalgebra
$L_n(\go\gs\gp_{m|2r})$.
\end{enumerate}
\end{cor}

\begin{proof}  
For the first statement, recall that
$\text{Com}(\tilde{V}^{-m+r}(\gg\gl_n), \cS(nm) \otimes \cE(nr)) \cong \tilde{V}^n(\gs\gl_{r|m})$. Applying Theorem \ref{cosetsimplicity}, the condition that $n\geq 2$ and  $-m + r +n  > 0$, implies that $$\text{Com}(\tilde{V}^{-m+r}(\gs\gl_n) ,  \cS(nm) \otimes \cE(nr)),$$ is simple. We have 
$\tilde{V}^{-m+r}(\gg\gl_n) \cong \tilde{V}^{-m+r}(\gs\gl_n) \otimes \cH$ where $\cH$ is a rank one Heisenberg algebra whenever $ m \neq r$. The simplicity of 
$$\text{Com}(\tilde{V}^{-m+r}(\gg\gl_n) ,  \cS(nm) \otimes \cE(nr)) = \text{Com}(\cH, \text{Com}(\tilde{V}^{-m+r}(\gs\gl_n) ,  \cS(nm) \otimes \cE(nr))) \cong \tilde{V}^n(\gs\gl_{r|m}),$$ then follows from \cite[Prop. 3.2]{CKLR}.

For the second statement, since $\text{Com}(\tilde{V}^{-\frac{m}{2}+r}(\gs\gp_{2n}),  \cS(nm) \otimes \cE(2nr)) \cong \tilde{V}^{n}(\go\gs\gp_{m|2r})$, Theorem \ref{cosetsimplicity} again implies that $ \tilde{V}^{n}(\go\gs\gp_{m|2r})$ is simple when $-\frac{m}{2} + r +n+1 > 0$.
\end{proof}

\section{Main result}

\begin{thm} \label{classfree} For all positive integers $m,n,r$, $\tilde{V}^n(\go\gs\gp_{m|2r})$ is classically free.
\end{thm}

\begin{proof} Let $$f:  \tilde{V}^{n}(\go\gs\gp_{m|2r}) \rightarrow \cS(nm) \otimes \cE(2nr)$$ be the map given by \eqref{s2}, and recall that 
$$\tilde{V}^{n}(\go\gs\gp_{m|2r}) \cong \text{Com}(\tilde{V}^{-\frac{m}{2}+r}(\gs\gp_{2n}),  \cS(nm) \otimes \cE(2nr)) = (\cS(nm) \otimes \cE(2nr))^{\gs\gp_{2n}[t]}.$$
Next, recall that $\cS(nm) \otimes \cE(2nr)$ has a good increasing filtration $G_{\bullet} (\cS(nm) \otimes \cE(2nr))$, where the space $G_{\frac{p}{2}} (\cS(nm) \otimes \cE(2nr))$ of degree at most $\frac{p}{2}$ is spanned by all monomials in the generators $\beta, \gamma, b,c$ and their derivatives, of length at most $p$. Also, $\tilde{V}^{n}(\go\gs\gp_{m|2r})$ has a good increasing filtration $G_{\bullet}( \tilde{V}^{n}(\go\gs\gp_{m|2r}))$, where the space $G_{p} (\tilde{V}^{n}(\go\gs\gp_{m|2r}))$ of degree at most $p$ is spanned by all monomials in the generators and their derivatives, of length at most $p$. Then $f$ preserves the filtration, that is, $f(G_p (\tilde{V}^{n}(\go\gs\gp_{m|2r}) )\subseteq G_p( \cS(nm) \otimes \cE(2nr))$; this is because the generators of $\tilde{V}^{n}(\go\gs\gp_{m|2r})$ are quadratics in the generators $\beta, \gamma, b,c$ of $\cS(nm) \otimes \cE(2nr)$. We have the induced good increasing filtration $G^{f}_{\bullet} ( \tilde{V}^{n}(\go\gs\gp_{m|2r}))$ given by  
\begin{equation} \label{compatiblefilt} G^{f}_{p}(\tilde{V}^{n}(\go\gs\gp_{m|2r}))= \{ \alpha \in \tilde{V}^{n}(\go\gs\gp_{m|2r})|\  f (\alpha) \in G_p (\cS(nm) \otimes \cE(2nr))\}.\end{equation} We denote by $\text{gr}_f( \tilde{V}^{n}(\go\gs\gp_{m|2r})))$ and $\text{gr}_G( \tilde{V}^{n}(\go\gs\gp_{m|2r})))$ the associated graded algebras with respect to these filtrations.

As in \cite{LS1}, we have the induced injective map 
\begin{equation} \label{surjectivity} \begin{split} & \text{gr}_f(\cS(nm) \otimes \cE(2nr))^{\gs\gp_{2n}[t]}) \hookrightarrow (\text{gr}_G(\cS(nm) \otimes \cE(2nr)))^{\gs\gp_{2n}[t]} 
\\ & \cong \big( \text{Sym}  \big(\bigoplus_{j\geq 0}V_j  \big) \bigotimes \bigwedge \big(\bigoplus_{j\geq 0} U_j\big)\big)^{J_{\infty}(Sp_{2n})},\end{split} \end{equation} where $V_j$ is isomorphic to $(\mathbb{C}^{2n})^{\oplus m}$ and $U_j \cong (\mathbb{C}^{2n})^{\oplus 2r}$, as $Sp_{2n}$-modules for all $j\geq 0$. We claim that \eqref{surjectivity} is surjective, and hence an isomorphism. To see this, note that the right hand side $\big(\text{Sym}  \big(\bigoplus_{j\geq 0}V_j  \big) \bigotimes \bigwedge \big(\bigoplus_{j\geq 0} U_j\big)\big)^{J_{\infty}(Sp_{2n})}$ is generated as a differential algebra by the subspace $\big((\text{Sym} \ V_0)  \bigotimes (\bigwedge U_0) \big)^{Sp_{2n}}$, by \cite[Thm. 3.1 (1)]{LS1} together with \cite[Thm. 1.4]{LS2}. Moreover, $\big((\text{Sym} \ V_0)  \bigotimes (\bigwedge U_0) \big)^{Sp_{2n}}$ is generated by quadratics corresponding to the pairing between two copies of the standard representation $\mathbb{C}^{2n}$ appearing in either $V_0$ or $U_0$, which must be distinct if both appear in $V_0$. Regarded as elements of $$\text{Sym}  \big(\bigoplus_{j\geq 0}V_j  \big) \bigotimes \bigwedge \big(\bigoplus_{j\geq 0} U_j\big) \cong \text{gr}_G(\cS(nm) \otimes \cE(2nr)),$$ these quadratics lift to quadratic fields in $\cS(nm) \otimes \cE(2nr)$; we simply replace the associative product by the normally ordered product. In fact, these quadratic fields are precisely the generators of $\tilde{V}^{n}(\go\gs\gp_{m|2r})$; this is the key step in the proof of \cite[Thm. 6.3]{LS1} which is omitted for brevity but is the same as the proof of  \cite[Thm. 4.4]{LS1}. But since $\tilde{V}^{n}(\go\gs\gp_{m|2r})$ commutes with $\tilde{V}^{-\frac{m}{2}+r}(\gs\gp_{2n})$, the above quadratic fields actually lie in $(\cS(nm) \otimes \cE(2nr))^{\gs\gp_{2n}[t]}$. Therefore the generators of $\big((\text{Sym} \ V_0)  \bigotimes (\bigwedge U_0) \big)^{Sp_{2n}}$ are the images under \eqref{surjectivity} of elements of $\text{gr}_f(\cS(nm) \otimes \cE(2nr))^{\gs\gp_{2n}[t]})$, so  \eqref{surjectivity} is surjective.

The hypotheses of \cite[Lem. 2.2]{LS1} are then satisfied, so we obtain
\begin{equation} \label{compatiblefilt} \text{gr}_f(\cS(nm) \otimes \cE(2nr))^{\gs\gp_{2n}[t]})  \cong \text{gr}_G(\cS(nm) \otimes \cE(2nr))^{\gs\gp_{2n}[t]}) \cong \text{gr}^F(\cS(nm) \otimes \cE(2nr))^{\gs\gp_{2n}[t]}). \end{equation} Here $\text{gr}^F(\cS(nm) \otimes \cE(2nr))^{\gs\gp_{2n}[t]})$ denotes the associated graded algebra with respect to Li's canonical decreasing filtration. It follows that we have an isomorphism of differential algebras
$$\text{gr}^F(\tilde{V}^n(\go\gs\gp_{m|2r})) \cong \text{gr}^F((\cS(nm) \otimes \cE(2nr))^{\gs\gp_{2n}[t]}) \cong \bigg( \text{Sym}  \big(\bigoplus_{j\geq 0}V_j  \big) \bigotimes \bigwedge \big(\bigoplus_{j\geq 0} U_j\big)\bigg)^{J_{\infty}(Sp_{2n})}.$$
Finally, since both $U_j$ and $V_j$ are isomorphic to sums of copies of the standard representation $\mathbb{C}^{2n}$ of $Sp_{2n}$, \cite[Thm. 3.1 (2)]{LS1} and \cite[Thm. 1.4]{LS2} imply that all differential algebraic relations among the generators of $\big( \text{Sym}  \big(\bigoplus_{j\geq 0}V_j  \big) \bigotimes \bigwedge \big(\bigoplus_{j\geq 0} U_j\big)\big)^{J_{\infty}(Sp_{2n})}$ and their derivatives, are consequences of the relations in $ \big( (\text{Sym} \ V_0)  \bigotimes (\bigwedge U_0) \big)^{Sp_{2n}}$ and their derivatives. The same statement then holds in $\text{gr}^F(\tilde{V}^n(\go\gs\gp_{m|2r}))$. Equivalently, $\tilde{V}^n(\go\gs\gp_{m|2r})$ is classically free.
\end{proof}

\begin{cor}  For $-\frac{m}{2} + r +n+1 > 0$, the simple affine vertex superalgebra $L_n(\go\gs\gp_{m|2r})$ is classically free. In particular, $L_n(\go\gs\gp_{1|2r})$ is classically free for all positive integers $n,r$.
\end{cor} 

\begin{proof} This is immediate from Corollary \ref{simplicityembeddings} and Theorem \ref{classfree}, since $\tilde{V}^n(\go\gs\gp_{m|2r}) = L_n(\go\gs\gp_{m|2r})$ in this case.
\end{proof}
In view of \cite[Thm. 5.5]{AL} and \cite[Thm. 7.1]{CL}, the algebras $L_n(\go\gs\gp_{1|2r})$ are new examples of lisse and rational vertex superalgebras which are classically free. 

The key ingredient in the proof of Theorem \ref{classfree} is the compatibility of the three filtrations on $\tilde{V}^{n}(\go\gs\gp_{m|2r})$, which is expressed by \eqref{compatiblefilt}. One can ask whether the classical freeness of $\tilde{V}^n (\gs\gl_{r|m})$ and $\tilde{V}^n(\gs\gl_r) = L_n(\gs\gl_r)$, which is just the case $m=0$, can be proved in a similar way. Unfortunately, this does not work because the filtration $G^f_{\bullet}(\tilde{V}^n (\gs\gl_{r|m}))$ on $\tilde{V}^n (\gs\gl_{r|m})$ induced by the embedding $\tilde{V}^n (\gs\gl_{r|m}) \hookrightarrow \cS(nm) \otimes \cS(nr)$, is not the same as Li's canonical increasing filtration $G_{\bullet}(\tilde{V}^n (\gs\gl_{r|m}))$. It is still expected that $ L_n(\gs\gl_r)$ is classically free for all positive integers $n,r$, but a different approach will be needed to prove this.

\end{document}